\documentclass[reprint,onecolumn,aip,amsmth,amssymb,nolongbibliography]{revtex4-2}
\usepackage{amsthm}
\usepackage{amsmath, mathtools}
\usepackage[utf8]{inputenc}	
\usepackage[english]{babel}
\usepackage{amsmath}
\usepackage{graphicx}		
\usepackage{natbib}
\usepackage{textcomp}
\usepackage{gensymb}
\usepackage[usenames,dvipsnames,svgnames,table]{xcolor}
\usepackage[hidelinks,colorlinks=false,urlcolor=Cerulean,citecolor=black]{hyperref}
\usepackage{siunitx}
\usepackage{mathrsfs}
\usepackage{multirow}
\usepackage{bm}
\usepackage{xfrac}
\usepackage{comment}
\usepackage[normalem]{ulem}
\usepackage{lineno}

\theoremstyle{plain}
\newtheorem{thm}{Theorem}[section]

\newtheorem{prop}[thm]{Proposition}

\theoremstyle{definition}

\newtheorem{rem}{Remark}

\makeatletter
\newcommand{\neutralize}[1]{\expandafter\let\csname c@#1\endcsname\count@}
\makeatother

\DeclareMathOperator{\Spec}{{\rm  Spec}}

    \DeclareFontFamily{U}{wncy}{}
    \DeclareFontShape{U}{wncy}{m}{n}{<->wncyr10}{}
    \DeclareSymbolFont{mcy}{U}{wncy}{m}{n}
    \DeclareMathSymbol{\Sha}{\mathord}{mcy}{"58} 

\numberwithin{equation}{section}

\renewcommand{\i}{\mathrm{i}}

\DeclareSymbolFont{bbold}{U}{bbold}{m}{n}
\DeclareSymbolFontAlphabet{\mathbbold}{bbold}

\newcommand{\revision}[1]{{\color{black}#1}}

\newcommand{\revisiontwo}[1]{{\color{black}#1}}

\makeatletter
\newcommand{\tpmod}[1]{{\@displayfalse\pmod{#1}}}
\makeatother

\begin{document}

\title{\revision{Composed solutions of synchronized patterns in multiplex networks of Kuramoto oscillators}}

\author{Priya B. Jain}
\thanks{equal contribution}
\author{Tung T. Nguyen}
\thanks{equal contribution}
\author{J\'an Min\'a{\v c}}
\author{Lyle E. Muller}
\author{Roberto C. Budzinski}
\email{rbudzins@uwo.ca}
\affiliation{Department of Mathematics, Western University, London, N6A 3K7 ON, Canada}
\affiliation{Western Institute for Neuroscience, Western University, London, N6A 3K7 ON, Canada}
\affiliation{Western Academy for Advanced Research, Western University, London, N6A 3K7 ON, Canada}

\begin{abstract}
Networks with different levels of interactions, including multilayer and multiplex networks, can display a rich diversity of dynamical behaviors and can be used to model and study a wide range of systems. Despite numerous efforts to investigate these networks, obtaining mathematical descriptions for the dynamics of multilayer and multiplex systems is still an open problem. Here, we combine ideas and concepts from linear algebra and graph theory with nonlinear dynamics to offer a novel approach to study multiplex networks of Kuramoto oscillators. Our approach allows us to study the dynamics of a large, multiplex network by decomposing it into two smaller systems: one representing the connection scheme within layers \revision{(intra-layer)}, and the other representing the connections between layers \revision{(inter-layer)}. \revision{Particularly, we use this approach to compose solutions for multiplex networks of Kuramoto oscillators. These solutions are given by a combination of solutions for the smaller systems given by the intra and inter-layer system and, in addition, our approach allows us to study the linear stability of these solutions.}
\end{abstract}

\maketitle

\begin{quotation}
Networks of nonlinear oscillators offer the possibility to model and study many natural systems. The pattern of connections and the coupling structure play a crucial role in the emergent dynamics in these systems. In this context, multilayer and multiplex networks of nonlinear oscillators depict a rich diversity of synchronization patterns. At the same time, the sophisticated connectivity patterns in these systems bring an intrinsic difficulty to the mathematical analyses of the dynamics. Here, we introduce a mathematical approach for multiplex networks of nonlinear oscillators where we can compose solutions with nontrivial patterns of oscillations and study their linear stability.
\end{quotation}

\section{Introduction}\label{sec:introduction}

Systems composed of coupled units have been used to model and study a diversity of phenomena in nature spanning from physics \cite{boccaletti2002synchronization,arenas2008synchronization} and engineering \cite{motter2013spontaneous,schafer2018dynamically}, to social science \cite{jusup2022social,borgatti2009network}, to biology \cite{banerjee2016chimera,perez2020ecology} and neuroscience \cite{deco2009key,bassett2018nature}. In this context, many systems have different levels of interactions, which can be understood as multilayer networks \cite{kivela2014multilayer,de2013mathematical}. In this case, the whole system can be understood as the composition of an internal level, within each layer, and an external level, between layers. This class of system can be visualized as a network of networks, and it has many direct applications \cite{salehi2015spreading,menichetti2016control,de2016physics,aleta2019multilayer,bassett2017network,yuvaraj2021topological,medeiros2021asymmetry}. A particular example of this kind of network is given by multiplex networks, which has received great attention in the past years \cite{gomez2013diffusion,battiston2014structural,gomez2012evolution,nicosia2015measuring,kouvaris2015pattern}. A multiplex network can be understood as a network with many layers, where each layer has the same number of nodes connected through a given internal connection scheme, and the connection between nodes in different layers is given by a one-to-one scheme, where a node in a given layer is connected to nodes in neighboring layers that are in the same relative position within the layer. Multiplex networks have been studied in many different contexts, where rich dynamics have been found \cite{gambuzza2015intra,leyva2018relay,leyva2017inter,singh2015synchronization,schulen2021solitary}. However, despite the efforts in past years and the advances on the investigation of this class of networks, there are many open questions, mainly regarding mathematical and analytical approaches to study the dynamics of multilayer and multiplex networks.

Multilayer and multiplex networks can display a great diversity of synchronization phenomena. For instance, first-order transition, or explosive synchronization has been reported in these networks \cite{kumar2020interlayer,nicosia2017collective,jalan2019inhibition}. Furthermore, different synchronization patterns, including chimera states, have been observed  \cite{ghosh2016emergence,sawicki2018delay,ghosh2016birth,rybalova2021interplay}. In this paper, we introduce an approach to study multiplex networks, where we leverage recent results from graph theory and linear algebra \cite{doan2022joins}. We recently proposed a mathematical approach to study the dynamical behavior of oscillators on multilayer networks where each node in a given layer is connected to all other oscillators in the neighbor layers \cite{nguyen2023broadcasting}. \revision{In this paper, we use similar ideas to introduce a novel approach to study multiplex networks. Differently from the previous paper, here, we consider a network of networks where each node in a given layer is connected to only one node in the neighboring layers. Further, our results can be applied to networks with heterogeneous intrinsic properties for different nodes.} The approach that we explore here considers a multiplex network as a decomposition into the intra-layer and the inter-layer structures. We remark that similar ideas have been proposed for different multi-level systems \cite{kivela2014multilayer,boccaletti2014structure,gao2012networks}. For instance, a related approach in this context is explored in \cite{berner2021multiplex}, where multilayer networks are decomposed, which allows for the studying of the master stability function of these systems with application to spiking neural networks.

Our article focuses on multiplex networks of nonlinear oscillators. The dynamics on these networks are described by the Kuramoto model, a traditional dynamical system used to study many synchronization phenomena \cite{acebron2005kuramoto,rodrigues2016kuramoto,strogatz2000kuramoto}. Multilevel systems of Kuramoto oscillators have been extensively studied in the past years, where a rich diversity of dynamics has been observed \cite{kumar2021explosive,jalan2019explosive,frolov2018macroscopic,khanra2018explosive,zhang2015explosive}. The mathematical framework we introduce here gives us novel insights into the dynamics of multiplex Kuramoto networks, which can now be studied in simpler terms. Here, multiplex networks are composed of $M$ layers with $N$ oscillators in each one. In this case, our approach indicates that, instead of studying a large system, composed of $MN$ units with different levels of interaction, we can decompose the multiplex network into two smaller systems: the "intra-layer" system composed of $N$ units; and the "inter-layer" system composed of $M$ units. With this, we can study the dynamics of these smaller systems to obtain insights into the dynamics of the whole system. Particularly, our framework allows us to obtain the trajectories of Kuramoto oscillators on a multiplex network by only studying the smaller systems. Further, our approach allows us to use the solutions for the intra and inter-layer systems to compose solutions for the multiplex one, which offers a new perspective on the equilibrium points for multiplex networks of nonlinear oscillators. We can also obtain the collective dynamics of the multiplex network, i.e. the Kuramoto order parameter, by using the same idea. Lastly, this approach allows us to obtain insights into the linear stability of the solutions on the multiplex network by analyzing the spectral properties of the matrices related to the smaller systems. 

Here, we first introduce a new perspective on multiplex networks using certain constructions in graph theory (Sec. \ref{sec:product}). We then discuss the Kuramoto model and our approach for Kuramoto oscillators on multiplex networks (Sec. \ref{sec:kuramoto_multiplex}), which allows us to study the dynamics of multiplex networks in simpler terms (Sec. \ref{sec:composed solution}). We extend this approach to the analysis of the stability of equilibrium points in multiplex networks (Sec. \ref{sec:stability_analysis}). At last, we display several numerical simulations of Kuramoto oscillators on multiplex networks, highlighting the diversity of dynamical behavior they show and the applicability of our approach (Sec. \ref{sec:simulations}). The discussions and conclusions are in Sec. \ref{sec:discussion_conclusion}, and all computational details can be found in the appendix.

\section{Kronecker sum and representation of a multiplex network} \label{sec:product}

Our approach focuses on multiplex networks, i.e. the connection between nodes in different layers is given by a one-to-one scheme, where the node $i$ in layer $l$ is connected to node $i$ in layer $k$. A schematic example is shown in Fig. \ref{fig:example_multiplex} where two layers are considered.
\begin{figure}[htb]
    \centering
    \includegraphics[width=0.45\textwidth]{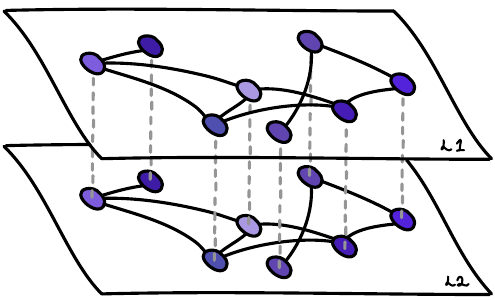}
    \caption{\textbf{Representation of a multiplex network.} In this kind of system, different layers are interconnected. Within a given layer, the oscillators have an "intra-layer" connection scheme; between layers, the oscillators are connected in a one-to-one scheme, which characterizes the "inter-layer" connection structure.}
    \label{fig:example_multiplex}
\end{figure}

A key observation here is that we can represent a multiplex network as the Cartesian product of two graphs. Let us first recall this concept from graph theory. Let $G$ and $H$ be two graphs. The Cartesian product operation forms a graph $G\boxtimes H$ of order $MN$ from a graph $G = (U, E)$ of order $M$ and a graph $H = (V, F)$ of order  $N$ \revision{ in the following way \cite{CRS}:}
\begin{itemize}
    \item \revision{The vertex set of $G\boxtimes H$ are ordered pairs $(u, v)$, where $u\in U$ and $v\in V$.}
    \item \revision{The vertices $(u, v)$ and $(u',v')$ are connected when:}
    \begin{itemize}
        \item \revision{$u = u'$ and $v$ and $v'$ are connected in $H$, or}
        \item \revision{$v = v'$ and $u$ and $u'$ are connected in $G$.}
    \end{itemize}
\end{itemize}
In other words, $G\boxtimes H$ is formed by replacing each vertex of $G$ by a copy of $H$, and replacing each edge of $G$ by edges between corresponding vertices of the appropriate copies. 

With this perspective, we can represent the adjacency matrix of $G \boxtimes H$ in a rather explicit way using some concepts from linear algebra and matrix theory which we now recall. For any positive integers $p, q, r, s$ we
define the Kronecker product of two matrices $\bm{A} \in \mathbb{R}^{p\times q}$ and $\bm{B} \in \mathbb{R}^{r\times s}$ as a
matrix $\bm{C} \in \mathbb{R}^{pr\times qs}$ given in block form as
\begin{equation}
\bm{C}=\left(\begin{array}{c|c|c|c}
\bm{A}b_{11} & \bm{A}b_{12} & \cdots & \bm{A}b_{1s} \\
\hline
\bm{A}b_{21} & \bm{A}b_{22} & \cdots & \bm{A}b_{2s} \\
\hline
\vdots & \vdots & \ddots & \vdots \\
\hline
\bm{A}b_{r1} & \bm{A}b_{r2} & \cdots & \bm{A}b_{rs}
\end{array}\right),
\label{eq:matrix_Kronecker_product}
\end{equation}
where $\bm{A}=[a_{ij}]$ and $\bm{B}=[b_{ij}]$. We denote the Kronecker product of $\bm{A}$ and $\bm{B}$ by $\bm{C}=\bm{A}\otimes \bm{B}$. For positive integers $r, s, k$, let $\bm{A} \in \mathbb{R}^{r \times r}$, $\bm{B} \in \mathbb{R}^{s \times s}$, and $\bm{I}_k$ be the identity matrix of order $k$. The sum $\bm{A} \otimes \bm{I}_s + \bm{I}_r \otimes \bm{B}$ is known as the Kronecker sum of $\bm{A}$ and $\bm{B}$. We denote the Kronecker sum of $\bm{A}$ and $\bm{B}$ by $\bm{A} \bigoplus \bm{B}$. 

By definition, the adjacency matrix of $G\boxtimes H$ is $\bm{A}\bigoplus \bm{B}$, where $\bm{A}$ is the adjacency matrix of $G$ and $\bm{B}$ is the adjacency matrix of $H$ \cite{CRS}. It is well-known that graph spectra behave well under the Cartesian product. Specifically, the spectrum of $\bm{A}\bigoplus \bm{B}$ is $\{\alpha_i + \beta_j \}_{1 \leq i \leq M, 1 \leq j \leq N}$ where $\{\alpha_i \}_{i=1}^r$ is the spectrum of $\bm{A}$  and $\{\beta_j \}_{j=1}^s$ is the spectrum of $\bm{B}$. \revision{Namely,}
\begin{equation}
\revision{
    \Spec\big(\bm{A}\bigoplus \bm{B}\big) =  \Spec(\bm{A})+\Spec(\bm{B}).}
    \label{eq:kronecker_sum_spec}
\end{equation}
Furthermore, the eigenvectors of $\bm{A}\bigoplus \bm{B}$ associated with ${\alpha_i + \beta_j}$ are Kronecker products of the corresponding eigenvectors of $\bm{A}$ and $\bm{B}$ \cite{CRS}. \revision{We note that similar graph theory approaches have been considered in the context of networked systems \cite{berner2021multiplex, CRS, nguyen2023broadcasting}. In our paper, we use these ideas to further develop the study of multiplex networks of Kuramoto oscillators.}

\section{Kuramoto oscillators on multiplex networks}\label{sec:kuramoto_multiplex}

In this paper, we focus on networks of nonlinear oscillators, where the dynamical behavior of each node is given by the Kuramoto model  \cite{rodrigues2016kuramoto,acebron2005kuramoto}
\begin{equation}
\frac{d\theta_i(t)}{dt} = \revision{\omega} + \sum_{j=1}^{\mathcal{N}} \mathbf{K}_{ij} \sin( \theta_j(t) - \theta_i(t) ).
\label{eq:KM_main}
\end{equation}
Here, $\theta_{i}(t) \in [-\pi, \, \pi] $ is the phase of the $i^{\mathrm{th}}$ oscillator at time $t$, \revision{$\omega$} is its natural frequency, $\mathcal{N}$ is the number of oscillators in the system, and \revision{ $\mathbf{K}_{ij}$ are elements of the adjacency matrix representing the connection strength from the $i^{\mathrm{th}}$ oscillator to the $j^{\mathrm{th}}$ oscillator. We focus on multiplex networks on which the intra-layer coupling strength is independent of the layer and assume that the natural frequency is constant for all the oscillators in the system.} 

\revision{For multiplex networks}, we can divide the connection structure into two different categories: intra-layer (within each layer) and inter-layer (between layers). \revision{Let us consider a system with $M$ layers and $N$ oscillators in each layer. For $l,k \in [1, \, M]$, we define $N \times N$ matrices $\bm{A}_{lk}$ whose $ij^{\mathrm{th}}$ entry $(\bm{A}_{lk})_{ij}$ is the connection strength from the $i^{\mathrm{th}}$ oscillator in layer $l$ to the $j^{\mathrm{th}}$ oscillator in layer $k$. In particular, the matrices $\bm{A}_{ll}$ represent the intra-layer coupling strength connection between oscillators in layer $l$ and the matrices $\bm{A}_{lk}$ for $l$ $\neq$ $k$ represent the inter-layer connection strength from layer $l$ to layer $k$. We also define a vector}
\begin{equation}
\revision{\bm{\theta}_l(t) = \big[(\bm{\theta}_{l})_{1}(t), (\bm{\theta}_{l})_{2}(t), \cdots, (\bm{\theta}_{l})_{N}(t)\big],}
\end{equation}
\revision{where $(\bm{\theta}_{l})_{i}(t)$ gives the phase of the $i^{\mathrm{th}}$ oscillator in the  $l^{\mathrm{th}}$ layer at time $t$.} 
With this in mind, we can now write the equation for the dynamics of each node \revision{ in a system } as:
\begin{equation}
    \frac{d(\bm{\theta}_{l})_{i}(t)}{dt} = \revision{\omega} + \underbrace{ \sum_{j=1}^{N} (\bm{A}_{ll})_{ij} \sin( (\bm{\theta}_{l})_{j}(t) - (\bm{\theta}_{l})_{i}(t) )}_{\mathrm{intra-layer}} + \underbrace{\sum_{k=1, k\neq l}^{M}\sum_{j=1}^{N} (\bm{A}_{lk})_{ij} \sin( (\bm{\theta}_{k})_{j}(t) - (\bm{\theta}_{l})_{i}(t) )}_{\mathrm{inter-layer}},
    \label{eq:multiplex_kuramoto}
\end{equation}
\revision{where $\omega$ is the natural frequency of all oscillators.}

We can still represent the multiplex network using a similar form \revision{as} in Eq. (\ref{eq:KM_main}). In this case, the information about intra-layer coupling and inter-layer coupling is represented in a single \revision{$NM \times NM$ matrix $K$ such that}
\begin{equation}
\revision{\bm{K}=\left(\begin{array}{c|c|c|c}
\bm{A}_{11} & \bm{A}_{12} & \cdots & \bm{A}_{1M} \\
\hline
\bm{A}_{21} & \bm{A}_{22} & \cdots & \bm{A}_{2M} \\
\hline
\vdots & \vdots & \ddots & \vdots \\
\hline
\bm{A}_{N1} & \bm{A}_{N2} & \cdots & \bm{A}_{NM}
\end{array}\right),}
\label{eq:matrix_multiplex}
\end{equation}
\revision{where the matrices $\bm{A}_{lk}$ are the same as introduced in Eq. (\ref{eq:multiplex_kuramoto}). Then the equation of the system turns out to be}
\begin{equation}
\frac{d\theta_i}{dt} = \omega + \sum_{j=1}^{NM} \mathbf{K}_{ij} \sin( \theta_j - \theta_i ),
\label{eq:KM_multiplex}
\end{equation}
where now $i \in [1,\, NM]$, such that:
\begin{equation}
    \bm{\theta} = (\underbrace{\theta_{1}, \theta_{2}, \cdots, \theta_{N}}_{1^{\mathrm{st}} \mathrm{layer}}, \underbrace{\theta_{N+1}, \theta_{N+2}, \cdots, \theta_{2N}}_{2^{\mathrm{nd}} \mathrm{layer}}, \cdots, \underbrace{\theta_{N(M-1) + 1}, \theta_{N(M-1) +2}, \cdots, \theta_{NM}}_{M^{\mathrm{th}} \mathrm{layer}}).
    \label{eq:theta_multiplex}
\end{equation}

\revision{The network being multiplex forces the off-diagonal blocks $\bm{A}_{lk}$ for $l \neq k$ to be scalar matrices. In particular, they are of the form}
\begin{equation}
\revision{
   \bm{A}_{lk} = \epsilon_{lk}I_{N}},
   \label{eq:matrix_off-diag_block}
\end{equation}
\revision{ where $\epsilon_{lk}$ represents the coupling strength between oscillators in layers $l$ and $k$ and $I_{N}$ is the $N \times N$ identity matrix.} This assumption means that, for any pair of layers, each oscillator in the first layer is only connected with the corresponding oscillator of the other layer. \revision{Moreover, since we have assumed that the intra-layer coupling strength is independent of the layers, all the diagonal blocks $\bm{A}_{ll}$ are identical. So, for all layers $l$, we have}
\begin{equation}
    \revision{\bm{A}_{ll}= }
    \revision{
\begin{pmatrix}
0 & a_{12} & \cdots & a_{1N} \\
a_{21} & 0 & \cdots & a_{2N} \\
\vdots  & \vdots  & \ddots & \vdots  \\
a_{N1} & a_{N2} & \cdots & 0 
\end{pmatrix}},
\label{eq:matrix_diag_block}
\end{equation}
\revision{where $a_{ij}$ represents the weight of the connection from oscillators $i$ to oscillator $j$ in each layer and no self-connections are considered.}

We can compose solutions for multiplex networks of Kuramoto oscillators using the ideas explained in the previous sections, combined with results in graph theory \cite{doan2022joins, CRS}. This is valid for equilibrium points and also for considering the transient behavior.
\revision{To do so, we define two new systems of Kuramoto oscillators representing the intra-layer and inter-layer connections. The intra-layer system consists of $N$ oscillators with equation:}
\begin{equation}
\revision{\frac{d\psi_i(t)}{dt} = \omega_{\mathrm{intra}} + \sum_{j=1}^{N} (\bm{A}_{\mathrm{intra}})_{ij} \sin( \psi_j(t) - \psi_i(t) ),}
\label{eq:KM_intra}
\end{equation}
\revision{where $\omega_{\mathrm{intra}}$ represents the natural frequency for this system and} 
\begin{equation}
    \revision{\bm{A}_{\mathrm{intra}}= 
\begin{pmatrix}
0 & a_{12} & \cdots & a_{1N} \\
a_{21} & 0 & \cdots & a_{2N} \\
\vdots  & \vdots  & \ddots & \vdots  \\
a_{N1} & a_{N2} & \cdots & 0 
\end{pmatrix},}
\label{eq:matrix_intra_system}
\end{equation}
\revision{where $a_{ij}$ are as in Eq. (\ref{eq:matrix_diag_block}). On the other hand, the inter-layer system consists of $M$ oscillators with equation}
\begin{equation}
\revision{\frac{d\phi_l(t)}{dt} = \omega_{\mathrm{inter}} + \sum_{k=1}^{M} (\bm{A}_{\mathrm{inter}})_{lk} \sin( \phi_k(t) - \phi_l(t) ),}
\label{eq:KM_inter}
\end{equation}
\revision{where $\omega_{\mathrm{inter}}$ represents the natural frequency for this system and}
\begin{equation}
    \revision{\bm{A}_{\mathrm{inter}}= 
\begin{pmatrix}
0 & \epsilon_{12} & \cdots & \epsilon_{1M} \\
\epsilon_{21} & 0 & \cdots & \epsilon_{2M} \\
\vdots  & \vdots  & \ddots & \vdots  \\
\epsilon_{1M} & \epsilon_{2M} & \cdots & 0 
\end{pmatrix},}
\label{eq:matrix_inter_system}
\end{equation}
\revision{where $\epsilon_{lk}$ are as in Eq. (\ref{eq:matrix_off-diag_block}).}

As explained in the previous section, we can now use the Kronecker sum of these two matrices and represent all connections between oscillators in the multiplex system as:
\begin{equation}
\revision{\bm{K}} = \bm{A}_{\mathrm{intra}} \bigoplus \bm{A}_{\mathrm{inter}} =
\left(\begin{array}{c|c|c|c}
\bm{A}_{\mathrm{intra}} & \epsilon_{12}\bm{I}_{N} & \cdots & \epsilon_{1M}\bm{I}_{N} \\
\hline
\epsilon_{21}\bm{I}_{N} & \bm{A}_{\mathrm{intra}} & \cdots & \epsilon_{2M}\bm{I}_{N} \\
\hline
\vdots & \vdots & \ddots & \vdots \\
\hline
\epsilon_{M1}\bm{I}_{N} & \epsilon_{M2}\bm{I}_{N} & \cdots & \bm{A}_{\mathrm{intra}}
\end{array}\right),\\ 
\label{eq:matrix_merged_multiplex}
\end{equation}
\revision{where $\bm{K}$ is as in Eq. (\ref{eq:matrix_multiplex})}. In addition to that, throughout the paper, we consider two different coupling strengths that act as scalars being multiplied by the matrices $\bm{A}_{\mathrm{intra}}$ and $\bm{A}_{\mathrm{inter}}$, which are $\epsilon_{\mathrm{intra}}$ and $\epsilon_{\mathrm{inter}}$, respectively. These parameters can be interpreted as the coupling strength for the intra-layer system and inter-layer system, which can be always absorbed into the adjacency matrices.

\section{Composition of solutions in multiplex networks}\label{sec:composed solution}

With the perspective described in the previous section, we can study a multiplex network composed of $M$ layers, each one with $N$ oscillators. In this framework, we can consider the case where the connectivity of each layer is identical and arbitrary. So, we can study the large and sophisticated multiplex system with $NM$ oscillators, by studying the behavior of simpler, smaller systems: a network with $N$ oscillators and connectivity given by the intra-layer connection scheme $\bm{A}_{\mathrm{intra}}$ as described in Eq. (\ref{eq:matrix_intra_system}), and a network with $M$ oscillators and connectivity given by the inter-layer connection scheme $\bm{A}_{\mathrm{inter}}$ as described in Eq. (\ref{eq:matrix_inter_system}). With this, we can obtain solutions and the transient dynamical behavior of the multiplex network by studying the behavior of the intra-layer and inter-layer representations.

This composition procedure can be summarized in the following proposition:
\begin{prop} 
\label{prop:composed solution}
Let 
\begin{equation}
    {\bm{\psi}}^{\ast} = ({\psi}_{1}^{\ast}, {\psi}_{2}^{\ast}, \cdots, {\psi}_{N}^{\ast}),
    \label{eq:solution_intra_system}
\end{equation} 
and
\begin{equation}
    {\bm{\phi}}^{\ast} = ({\phi}_{1}^{\ast}, {\phi}_{2}^{\ast}, \cdots, {\phi}_{M}^{\ast}),
    \label{eq:solution_inter_system}
\end{equation}
be solutions of two single layered networks given by matrices $\bm{A}_{\mathrm{intra}}$ and $\bm{A}_{\mathrm{inter}}$, whose dynamics is represented by Eqs. (\ref{eq:KM_intra}) and (\ref{eq:KM_inter}), respectively. Then, \revision{the Eq. (\ref{eq:multiplex_kuramoto})
with $\bm{A}_{ll}$ = $\bm{A}_{\mathrm{intra}}$ for all $l$, $\bm{A}_{lk}$ =  $(\bm{A}_{\mathrm{inter}})_{lk} \bm{I}_{N}$ and $\omega$ = $\omega_{\mathrm{inter}} + \omega_{\mathrm{intra}}$,}
\begin{equation}
\revision{
\begin{split}
    \frac{d(\bm{\theta}_{l})_{i}(t)}{dt} = (\omega_{\mathrm{inter}} + \omega_{\mathrm{intra}}) + \underbrace{ \sum_{j=1}^{N} (\bm{A}_{\mathrm{intra}})_{ij} \sin( (\bm{\theta}_{l})_{j}(t) - (\bm{\theta}_{l})_{i}(t) )}_{\mathrm{intra-layer}}\\ + \underbrace{\sum_{k=1, k\neq l}^{M}\sum_{j=1}^{N} ((\bm{A}_{\mathrm{inter}})_{lk} \bm{I}_{N})_{ij} \sin( (\bm{\theta}_{k})_{j}(t) - (\bm{\theta}_{l})_{i}(t) )}_{\mathrm{inter-layer}},
\end{split}
}
\end{equation}
 admits a solution 
\[
\bm{\theta^*}_l(t) = \big[(\bm{\theta^*}_{l})_{1}(t), (\bm{\theta^*}_{l})_{2}(t), \cdots, (\bm{\theta^*}_{l})_{N}(t)\big] \mathrm{ ,where} \  l = 1, 2, \cdots M,
\]

such that
\begin{equation}
(\bm{\theta}^{*}_{l})_{i} = \psi_{i}^{*} + \phi_{l}^{*}         
\label{eq:solution_composed_layer}
\end{equation}

Expressing in terms of Eq.(\ref{eq:KM_multiplex}) and (\ref{eq:theta_multiplex}),
\begin{equation}
    \bm{\theta}^{\ast} = (\underbrace{{\psi}_{1}^{\ast}+{\phi}_{1}^{\ast}, {\psi}_{2}^{\ast}+{\phi}_{1}^{\ast}, \cdots, {\psi}_{N}^{\ast}+{\phi}_{1}^{\ast}}_{1^{\mathrm{st}}\mathrm{layer}}, \underbrace{{\psi}_{1}^{\ast}+{\phi}_{2}^{\ast}, {\psi}_{2}^{\ast}+{\phi}_{2}^{\ast}, \cdots, {\psi}_{N}^{\ast}+{\phi}_{2}^{\ast}}_{2^{\mathrm{nd}} \mathrm{layer}}, \cdots, \underbrace{{\psi}_{1}^{\ast}+{\phi}_{M}^{\ast}, {\psi}_{2}^{\ast}+{\phi}_{M}^{\ast}, \cdots, {\psi}_{N}^{\ast}+{\phi}_{M}^{\ast}}_{M^{\mathrm{th}} \mathrm{layer}}),
\label{eq:solution_multiplex_composed}
\end{equation}
is a solution of the multiplex network \revision{ given by equation }
\begin{equation}
\revision{
\frac{d\theta_i}{dt} = (\omega_{\mathrm{inter}} + \omega_{\mathrm{intra}}) + \sum_{j=1}^{NM} \bm{K}_{ij} \sin( \theta_j - \theta_i ),}
\label{eq:KM_multiplex_composed}
\end{equation}
\revision{where }

$$\revision{\bm{K}} = \bm{A}_{\mathrm{intra}} \bigoplus \bm{A}_{\mathrm{inter}}.$$
\end{prop}

\begin{proof}
By definition, we are given that   
\begin{equation}
    \frac{d\psi_i^{\ast}}{dt} = \revision{\omega_{\mathrm{intra}}} + \sum_{j=1}^{N} (\bm{A}_{\mathrm{intra}})_{ij} \sin( \psi_j^{\ast} - \psi_i^{\ast} ),
    \label{eq:KM_intra_proof}
\end{equation}
and 
\begin{equation}
    \frac{d\phi_l^{\ast}}{dt} = \revision{\omega_{\mathrm{inter}} +} \sum_{k=1}^{M} (\bm{A}_{\mathrm{inter}})_{lk} \sin( \phi_k^{\ast} - \phi_l^{\ast} ).
    \label{eq:KM_inter_proof}
\end{equation}
We have 
\begin{equation}
    \frac{d(\psi_i^{\ast} +\phi_{l}^{\ast})}{dt} =  \frac{d\psi_i^{\ast}}{dt}+\frac{d\phi_l^{\ast}}{dt}.
\end{equation}
Here, we can use Eqs. (\ref{eq:KM_intra_proof}), (\ref{eq:KM_inter_proof}) and (\ref{eq:solution_composed_layer}), which leads to
\begin{equation}
  \revision{\frac{d((\bm{\theta}^{*}_{l})_{i})}{dt}} = \frac{d(\psi_i^{\ast} +\phi_{l}^{\ast})}{dt} = \revision{(\omega_{\mathrm{intra}} + \omega_{\mathrm{inter}}) +}\sum_{j=1}^{N} (\bm{A}_{\mathrm{intra}})_{ij} \sin {( \psi_j^{\ast} - \psi_i^{\ast} )}
+ \sum_{k=1}^{M} (\bm{A}_{\mathrm{inter}})_{lk} \sin {( \phi_k^{\ast} - \phi_l^{\ast} )}. 
\label{eq:proof_4}
\end{equation}
\revision{Using the fact that for $k \neq l$}
\[ \revision{((\bm{A}_{\mathrm{inter}})_{lk}\revision{I_{N}})_{ij} = \begin{cases}
    (\bm{A}_{\mathrm{inter}})_{lk} & \text{if } i = j  \\
   0& \text{else,}
\end{cases}} \]
\revision{we can see that the Eq. (\ref{eq:proof_4}) can be written as}
\begin{equation}
\begin{split}
    \revision{\frac{d((\bm{\theta}^{*}_{l})_{i})}{dt}} = \revision{(\omega_{\mathrm{intra}} + \omega_{\mathrm{inter}}) +} \underbrace{\sum_{j=1}^{N} (\bm{A}_{\mathrm{intra}})_{ij} \sin{\Big( (\psi_j^{\ast} - \psi_i^{\ast})+(\phi_{l}^{\ast}-\phi_{l}^{\ast})\Big)}}_{\mathrm{intra-layer}} + \\ \underbrace{\sum_{k=1, k\neq l}^{M}\sum_{j=1}^{N} ((\bm{A}_{\mathrm{inter}})_{lk}\revision{I_{N}})_{ij} \sin{\Big( (\phi_k^{\ast}-\phi_l^{\ast})
    + (\psi_j^{\ast}-\psi_j^{\ast} )\Big)}}_{\mathrm{inter-layer}},
    \end{split}
\end{equation}
\revision{where $\bm{I}_N$ be the identity matrix of order $N$. We can now write}
\begin{equation}
\begin{split}
    \revision{\frac{d((\bm{\theta}^{*}_{l})_{i})}{dt}} = \revision{(\omega_{\mathrm{intra}} + \omega_{\mathrm{inter}}) +} \underbrace{\sum_{j=1}^{N} (\bm{A}_{\mathrm{intra}})_{ij} \sin{\Big( (\psi_j^{\ast}+\phi_{l}^{\ast}) - (\psi_i^{\ast}+\phi_{l}^{\ast} )\Big)}}_{\mathrm{intra-layer}} + \\ \underbrace{\sum_{k=1, k\neq l}^{M}\sum_{j=1}^{N} ((\bm{A}_{\mathrm{inter}})_{lk}\revision{\bm{I}_{N}})_{ij} \sin{\Big( (\psi_j^{\ast}+\phi_k^{\ast})
    - (\psi_i^{\ast}+\phi_l^{\ast} )\Big)}}_{\mathrm{inter-layer}}.
\end{split}
\end{equation}
or
\begin{equation}
\begin{split}
   \revision{\frac{d((\bm{\theta}^{*}_{l})_{i})}{dt} =(\omega_{\mathrm{inter}} + \omega_{\mathrm{intra}})} + \underbrace{ \sum_{j=1}^{N} (\bm{A}_{\mathrm{intra}})_{ij} \sin( (\bm{\theta}_{l})_{j}(t) - (\bm{\theta}_{l})_{i}(t) )}_{\mathrm{intra-layer}}\\ + \underbrace{\sum_{k=1, k\neq l}^{M}\sum_{j=1}^{N} ((\bm{A}_{\mathrm{inter}})_{lk}\revision{\bm{I}_{N}})_{ij} \sin( (\bm{\theta}_{k})_{j}(t) - (\bm{\theta}_{l})_{i}(t) )}_{\mathrm{inter-layer}}.
\end{split} 
\end{equation}

With this, we prove that the solutions for the multiplex network are equivalent to the sum of the solutions for the intra and inter-layer systems.
\end{proof}

Furthermore, we can extend this analysis and characterize the dynamical behavior of a multiplex network using the Kuramoto order parameter, which quantifies in one single number the level of phase synchronization a given network has \cite{acebron2005kuramoto,rodrigues2016kuramoto}.
The Kuramoto order parameter for the multiplex system is defined as
\begin{equation}
    R(t) = \frac{1}{NM} \left|\sum\limits_{j=1}^{NM} \exp{(\i\theta_j(t))}\right|,
    \label{eq:order_parameter_multilayer}
\end{equation}
where $\theta(t)$ is given by Eq. (\ref{eq:KM_multiplex}). Here, $R(t) = 1$ means that all oscillators in all the layers have the same phase at a given time $t$, which is defined as phase synchronization. For the asynchronous behavior, $R$ assumes residual values. Using the same idea, we also measure the level of synchronization of the Kuramoto network on the inter-layer and intra-layer representations:
\begin{equation}
    R_{\mathrm{intra}}(t) = \frac{1}{N} \left|\sum\limits_{j=1}^{N} \exp{(\i \psi_j(t))}\right|,
    \label{eq:order_parameter_intra}
\end{equation}
and
\begin{equation}
    R_{\mathrm{inter}}(t) = \frac{1}{M} \left|\sum\limits_{j=1}^{M} \exp{(\i \phi_j(t))}\right|.
    \label{eq:order_parameter_inter}
\end{equation}

By the definition of the composed solution and direct calculations, we have the following proposition.
\begin{prop}\label{prop:order_parameter}
Suppose that $\bm{\psi}$ represents the behavior of the Kuramoto model on $\bm{A}_{\mathrm{intra}}$ -- given by Eq. (\ref{eq:KM_intra}) -- and $\bm{\phi}$ is a represents the behavior of the Kuramoto model on $\bm{A}_{\mathrm{inter}}$ -- given by Eq. (\ref{eq:KM_inter}). Further, $\bm{\theta}$ represents the dynamical behavior of the Kuramoto model on $\bm{K}$ -- the multiplex network which is represented by Eq. (\ref{eq:KM_multiplex}). Then the relation between their Kuramoto order parameters is as follows: 
\begin{equation}
 \revision{ R(t) = R_{\mathrm{intra}}(t) R_{\mathrm{inter}}(t).}
  \label{eq:order_parameter_composed}
\end{equation}

Here, $R_{\mathrm{intra}}(t)$ and $R_{\mathrm{inter}}(t)$ can be directly obtained through Eqs. (\ref{eq:order_parameter_intra}) and (\ref{eq:order_parameter_inter}), respectively.
\end{prop}

\begin{rem}
\label{rem:natural_frequency}
\revisiontwo{We remark that our approach can be also applied for multiplex networks with heterogeneous natural frequency. In this case, we consider that the natural frequency of each oscillator in the multiplex network depends on the layer that the oscillator belongs to. So the natural frequency of the inter-layer system $\omega_{\mathrm{inter}}$ becomes a vector with $M$ entries, where the $l^{\mathrm{th}}$ entry describes the natural frequency of oscillators on layer $l$. Lastly, we set $\omega_{\mathrm{intra}} = 0$ without loss of generality. With this, the propositions presented in the section can be directly applied to multiplex networks with heterogeneous natural frequency. Specifically, the proof of Prop. \ref{prop:composed solution} shows that if ${\bm{\psi}}^{\ast} = ({\psi}_{1}^{\ast}, {\psi}_{2}^{\ast}, \cdots, {\psi}_{N}^{\ast})$ is a solution of the intra-layer system given by Eq. (\ref{eq:KM_intra}) with $\omega_{\mathrm{intra}} = 0$, and ${\bm{\phi}}^{\ast} = ({\phi}_{1}^{\ast}, {\phi}_{2}^{\ast}, \cdots, {\phi}_{M}^{\ast})$ is a solution of the inter-layer system given by Eq. (\ref{eq:KM_inter}) with heterogeneous natural frequency $\bm{\omega}_{\mathrm{inter}} = \big((\omega_{\mathrm{inter}})_1, (\omega_{\mathrm{inter}})_2, \cdots, (\omega_{\mathrm{inter}})_M\big)$, then the composed solution given by Eq. (\ref{eq:solution_multiplex_composed}) is a solution of the Kuramoto model on the multiplex given by}
\begin{equation}
\begin{split}
\revisiontwo{
    \frac{d(\bm{\theta}_{l})_{i}(t)}{dt} = (\omega_{\mathrm{inter}})_{l} + \underbrace{ \sum_{j=1}^{N} (\bm{A}_{\mathrm{intra}})_{ij} \sin( (\bm{\theta}_{l})_{j}(t) - (\bm{\theta}_{l})_{i}(t) )}_{\mathrm{intra-layer}}}\\ \revisiontwo{ + \underbrace{\sum_{k=1, k\neq l}^{M}\sum_{j=1}^{N} ((\bm{A}_{\mathrm{inter}})_{lk} \bm{I}_{N})_{ij} \sin( (\bm{\theta}_{k})_{j}(t) - (\bm{\theta}_{l})_{i}(t) )}_{\mathrm{inter-layer}}.}
\end{split}
\end{equation}
\revisiontwo{This shows that Prop. \ref{prop:composed solution} and Prop. \ref{prop:order_parameter} can be used in the case of heterogeneous natural frequency.}
\end{rem} 

\section{Stability of composed solutions in multiplex networks}\label{sec:stability_analysis}

In order to perform stability analysis for a given solution, one needs a more in-depth analysis of the Jacobian matrix. In this section, we study the Jacobian for the intra \revision{and inter-layer networks} in order to obtain information about the Jacobian of the multiplex network. \revision{We use our approach to obtain the composed equilibrium point $\bm{\theta}^{\ast}$ by using the equilibrium point $\bm{\psi}^{\ast}$ of the intra-layer system -- given by Eq. (\ref{eq:solution_intra_system}) and Eq. (\ref{eq:KM_intra}) respectively and using the equilibrium point $\bm{\phi}^{\ast}$ of the inter-layer system -- given by Eq. (\ref{eq:solution_inter_system}) and Eq. (\ref{eq:KM_inter}) respectively. In this case, we denote the Jacobian for the intra-layer system at $\bm{\psi}^{\ast}$  as $\bm{J}_{\bm{A}_{\mathrm{intra}}}(\bm{\psi}^{\ast})$, for the inter-layer system at $\bm{\phi}^{\ast}$ as $\bm{J}_{\bm{A}_{\mathrm{inter}}}(\bm{\phi}^{\ast})$, and the Jacobian for the multiplex system at the equilibrium point $\bm{\theta}^{\ast}$ as $\bm{J}(\bm{\theta}^{\ast}) = \bm{J}_{\bm{K}}(\bm{\theta}^{\ast})$.}

We first recall the matrices $\bm{A}_{\mathrm{intra}}$ and $\bm{A}_{\mathrm{inter}}$, which are defined in Eq. (\ref{eq:matrix_intra_system}) and Eq. (\ref{eq:matrix_inter_system}) respectively. Based on these matrices, we can now write the Jacobian for \revision{the intra-layer system at $\bm{\psi}^{\ast}$} as
\begin{equation}
\bm{J}_{\bm{A}_{\mathrm{intra}}}({\bm{\psi}}^{\ast})=\begin{pmatrix}
-\lambda_1 & a_{12}\cos({\psi}^{\ast}_{2}-{\psi}^{\ast}_{1}) & \cdots & a_{1N}\cos({\psi}^{\ast}_{N}-{\psi}^{\ast}_{1}) \\
a_{21}\cos({\psi}^{\ast}_{1}-{\psi}^{\ast}_{2}) & -\lambda_2 & \cdots & a_{2N}\cos({\psi}^{\ast}_{N}-{\psi}^{\ast}_{2}) \\
\vdots  & \vdots  & \ddots & \vdots  \\
a_{N1}\cos({\psi}^{\ast}_{1}-{\psi}^{\ast}_{N}) & a_{N2}\cos({\psi}^{\ast}_{2}-{\psi}^{\ast}_{N}) & \cdots &  -\lambda_{N} 
\end{pmatrix},
\label{eq:jacobian_reduced system}
\end{equation}
where
\begin{equation}
    \lambda_i =\sum_{j=1, j\neq i}^N a_{ij}\cos({\psi}^{\ast}_{j}-{\psi}^{\ast}_{i}),
\end{equation}
\revision{and the Jacobian for inter-layer system at $\bm{\phi}^{\ast}$ as}

\begin{equation}
\revision{
\bm{J}_{\bm{A}_{\mathrm{inter}}}({\bm{\phi}}^{\ast})=\begin{pmatrix}
-\mu_1 & \epsilon_{12}\cos({\phi}^{\ast}_{2}-{\phi}^{\ast}_{1}) & \cdots & \epsilon_{1M}\cos({\phi}^{\ast}_{M}-{\phi}^{\ast}_{1}) \\
\epsilon_{21}\cos({\phi}^{\ast}_{1}-{\phi}^{\ast}_{2}) & -\mu_2 & \cdots & \epsilon_{2M}\cos({\phi}^{\ast}_{M}-{\phi}^{\ast}_{2}) \\
\vdots  & \vdots  & \ddots & \vdots  \\
\epsilon_{M1}\cos({\phi}^{\ast}_{1}-{\phi}^{\ast}_{M}) & \epsilon_{M2}\cos({\phi}^{\ast}_{2}-{\phi}^{\ast}_{M}) & \cdots &  -\mu_{M} 
\end{pmatrix},}
\label{eq:jacobian_inter system}
\end{equation}
\revision{where}
\begin{equation}
    \revision{\mu_l =\sum_{k=1, k\neq l}^M \epsilon_{lk}\cos({\phi}^{\ast}_{k}-{\phi}^{\ast}_{l}).}
\end{equation}
We now compute the Jacobian $\bm{J}(\bm{\theta}^{\ast})$. \revision{We express $\bm{J}(\bm{\theta}^{\ast})$ as a block matrix}
\begin{equation}
\revision{
\bm{J}(\bm{\theta}^{\ast}) = \bm{J}_{\bm{K}}(\bm{\theta}^{\ast}) = \left(\begin{array}{c|c|c|c}
\bm{J}_{11}(\bm{\theta}^{\ast}) & \bm{J}_{12}(\bm{\theta}^{\ast}) & \cdots & \bm{J}_{1M}(\bm{\theta}^{\ast}) \\
\hline
\bm{J}_{21}(\bm{\theta}^{\ast}) & \bm{J}_{22}(\bm{\theta}^{\ast}) & \cdots & \bm{J}_{2M}(\bm{\theta}^{\ast}) \\
\hline
\vdots & \vdots & \ddots & \vdots \\
\hline
\bm{J}_{N1}(\bm{\theta}^{\ast}) & \bm{J}_{N2}(\bm{\theta}^{\ast}) & \cdots & \bm{J}_{NM}(\bm{\theta}^{\ast})
\end{array}\right),}
\label{eq:jacobian_matrix_multiplex}
\end{equation}

\revision{By definition, we have for $l$ $\neq$ $k$,}
\begin{equation}
\revision{
(\bm{J}_{lk}(\bm{\theta}^{\ast}))_{ij} = (\bm{A}_{lk})_{ij} \cos{\Big(\psi_{j}^{\ast}+\phi_k^{\ast}-\psi_{i }^{\ast}-\phi_l^{\ast}\Big)} =  
(\epsilon_{lk}\bm{I}_{\mathrm{N}})_{ij}\cos{\Big(\psi_{i}^{\ast}+\phi_k^{\ast}-\psi_{i }^{\ast}-\phi_l^{\ast}\Big)} =
(\epsilon_{lk}\bm{I}_{\mathrm{N}})_{ij}  \cos{(\phi_{k}^{\ast}-\phi_{l}^{\ast})}.}
\end{equation}

Similarly, for $l$ = $k$ and $i \neq j$, we have 
\begin{align}
(\bm{J}_{lk}(\bm{\theta}^{\ast}))_{ij} = (\bm{J}_{ll}(\bm{\theta}^{\ast}))_{ij} &=  (\bm{A}_{ll})_{ij} \cos{\Big(\psi_{j}^{\ast}+\phi_{l}^{\ast}-\psi_{i }^{\ast}-\phi_l^{\ast}\Big)} = (\bm{A}_{\mathrm{intra}})_{ij} \cos{\Big(\psi_{j }^{\ast}-\psi_{i}^{\ast}\Big)}. 
\end{align} 

Finally, we need to consider the case $i=j$ \revision{in layer $l$}. For this part, we observe that $\bm{J}(\bm{\theta^}{\ast})$ is a semi-magic square matrix with a line sum equal to zero. We can then see that
\begin{equation}
\revision{
(\bm{J}_{ll}(\bm{\theta}^{\ast}))_{ii}= -\lambda_i-\mu_{l} .}
\end{equation}
By combining these facts, we can write the Jacobian for the multiplex system at the equilibrium point $\bm{\theta}^{\ast}$ as:
\begin{equation}
\label{eq:jacobian_multilayer_system} 
\bm{J}(\bm{\theta}^{\ast}) = \bm{J}_{\bm{K}}(\bm{\theta}^{\ast}) = \left(\begin{array}{c|c|c|c}
\revision{\bm{J}_{\bm{A}_{\mathrm{intra}}}({\bm{\psi}}^{\ast}) - \mu_1 \bm{I}_N} & (\epsilon_{12}\revision{\bm{I}_N})  \cos(\phi_{2}^{\ast}-\phi_{1}^{\ast}) & \cdots & (\epsilon_{1M} \revision{\bm{I}_N})  \cos(\phi_{M}^{\ast}-\phi_{1}^{\ast})\\
\hline
(\epsilon_{21} \revision{\bm{I}_N})  \cos(\phi_{1}^{\ast}-\phi_{2}^{\ast}) & \revision{\bm{J}_{\bm{A}_{\mathrm{intra}}}({\bm{\psi}}^{\ast}) - \mu_2 \bm{I}_N} & \cdots &(\epsilon_{2M} \revision{\bm{I}_N})  \cos(\phi_{M}^{\ast}-\phi_{2}^{\ast}) \\
\hline
\vdots & \vdots & \ddots & \vdots \\
\hline
(\epsilon_{M1} \revision{\bm{I}_N})  \cos(\phi_{1}^{\ast}-\phi_{M}^{\ast}) & (\epsilon_{M2} \revision{\bm{I}_N})  \cos(\phi_{2}^{\ast}-\phi_{M}^{\ast}) & \cdots &\revision{\bm{J}_{\bm{A}_{\mathrm{intra}}}({\bm{\psi}}^{\ast}) - \mu_M \bm{I}_N}
\end{array}\right)
\end{equation}
or
\begin{equation}
\revision{\bm{J}(\bm{\theta}^{\ast}) = \bm{J}_{\bm{K}}(\bm{\theta}^{\ast}) = \bm{J}_{\bm{A}_{\mathrm{intra}}}(\bm{\psi}^{\ast}) \bigoplus \bm{J}_{\bm{A}_{\mathrm{inter}}}(\bm{\phi}^{\ast}).}
\end{equation}
\revision{As a direct consequence of this equality, we have the following proposition.}

\begin{prop} \label{prop:spectrum}
The spectrum of $\bm{J}_{\bm K}(\bm{\theta}^{\ast})$ can be defined in terms of:
\begin{itemize}
    \item The spectra of the matrix, $\bm{J}_{\bm{A}_{\mathrm{inter}}}(\bm{\phi^{\ast}})$.
    \item The spectra of the Jacobian, $\bm{J}_{\bm{A}_{\mathrm{intra}}}(\bm{\psi^{\ast}})$. 
\end{itemize}
Namely, 
\begin{equation}
    \Spec(\bm{J}_{\bm{K}}(\bm{\theta}^{\ast})) =  \Spec(\bm{J}_{\bm{A}_{\mathrm{intra}}}({\bm \psi}^{\ast}))+\Spec(\bm{J}_{\bm{A}_{\mathrm{inter}}})({\bm \phi}^{\ast}))  .
\end{equation}
\end{prop}
\begin{proof}
\revision{This can be obtained using Eq. (\ref{eq:kronecker_sum_spec})} since $\bm{J}_{\bm{K}}(\bm{\theta}^{\ast})=\bm{J}_{\bm{A}_{\mathrm{intra}}}(\bm{\psi}^{\ast}) \bigoplus \bm{J}_{\bm{A}_{\mathrm{inter}}}(\bm{\phi}^{\ast})$.

\end{proof}

\begin{prop}\label{prop:linear stability}
The solution $\bm{\theta}^{\ast}$ for the multiplex network obtained through Eq. (\ref{eq:solution_multiplex_composed}) is linearly stable if and only if the solutions for the intra and inter-layer systems ${\bm {\psi}}^{\ast}$ and ${\bm \phi}^{\ast}$ are linearly stable. 
\end{prop}

\begin{proof}
Assume that ${\bm{\psi}}^{\ast}$ and ${\bm{\phi}}^{\ast}$are linearly stable, then $\bm{J}_{\bm{A}_{\mathrm{intra}}}({\bm{\psi}}^{\ast})$ and $\bm{J}_{\bm{A}_{\mathrm{inter}}}({\bm{\phi}}^{\ast})$ are symmetric negative-semidefinite. That is, all the eigenvalues of $\bm{J}_{\bm{A}_{\mathrm{intra}}}({\bm{\psi}}^{\ast})$ and $\bm{J}_{\bm{A}_{\mathrm{inter}}}({\bm{\phi}}^{\ast})$, except 0, are negative. So, all the eigenvalues of $\bm{J}_{\bm{K}}(\bm{\theta}^{\ast})$, except 0, are negative as the sum of negative numbers is negative. Hence, $\bm{\theta}^{\ast}$ is linearly stable.

Conversely, if the composed solution $\bm{\theta}^{\ast}$ is linearly stable then $0$ is an eigenvalue of $\bm{J}_{\bm{K}}(\bm{\theta}^{\ast})$ with multiplicity $1$ and all other eigenvalues must be negative. From this, we can conclude that all the eigenvalues of $\bm{J}_{\bm{A}_{\mathrm{intra}}}({\bm{\psi}}^{\ast})$ and $\bm{J}_{\bm{A}_{\mathrm{inter}}}({\bm{\phi}}^{\ast})$, except 0, are negative. This shows that ${\bm {\psi}}^{\ast}$ and ${\bm \phi}^{\ast}$ are linearly stable. 
\end{proof}

\section{Numerical simulations and applications}\label{sec:simulations}

\begin{figure*}[b]
    \centering
    \includegraphics[width=0.9\textwidth]{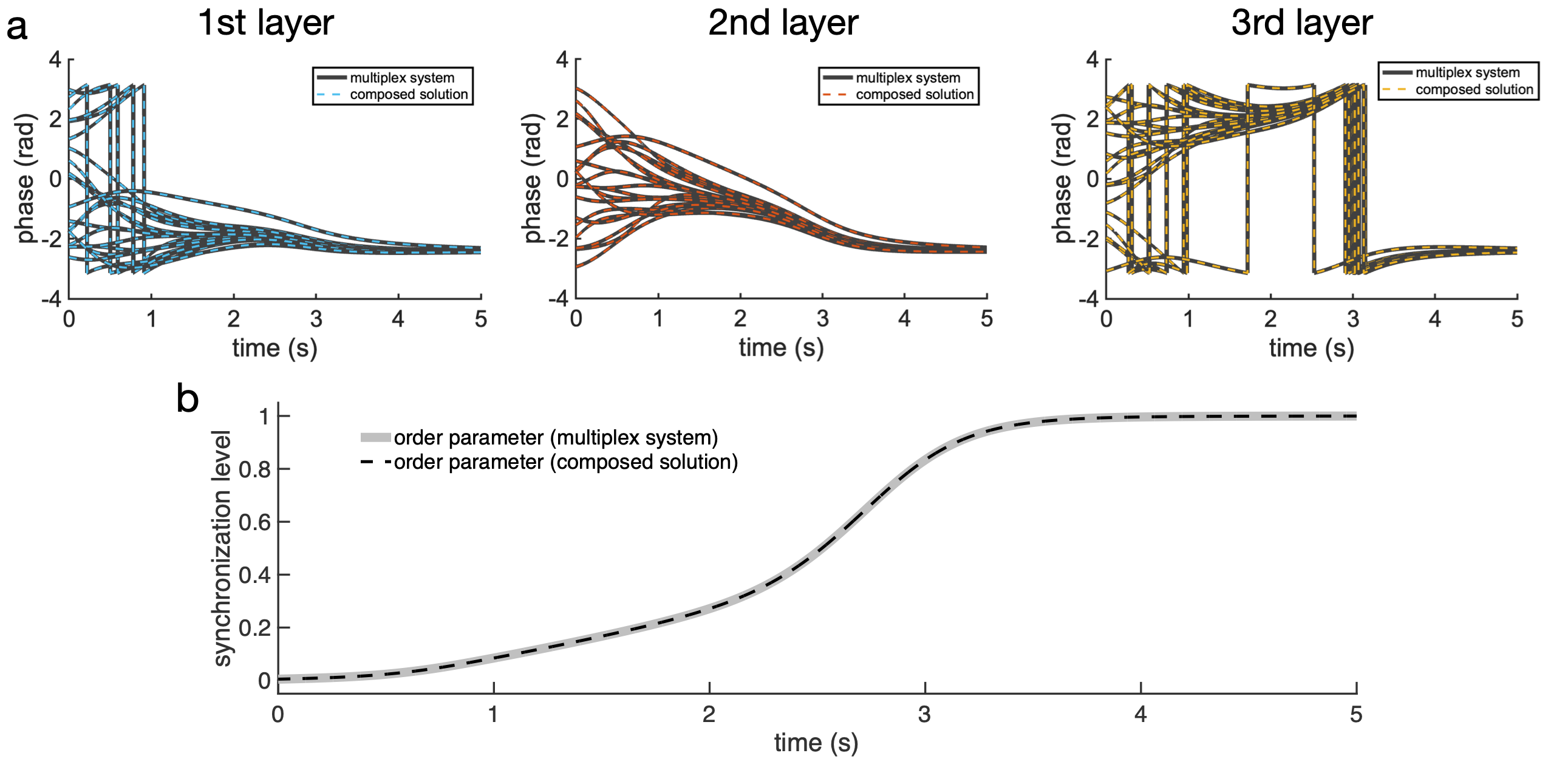}
    \caption{\textbf{The composed solution allows us to study the dynamics of the multiplex network.} We consider a multiplex network composed of $M = 3$ layers each one with $N = 20$ oscillators. The initial conditions for the oscillators are given by random phases $\mathcal{U}(-\pi,\pi)$. \textbf{(a)} We show the trajectories of the multiplex system for the oscillators in each layer (black solid lines), which are obtained through the integration of Eq. (\ref{eq:KM_multiplex}). These trajectories have a perfect match with the one obtained through the composed solution Eq. (\ref{eq:solution_multiplex_composed}), represented here by the colored dashed lines (equivalent system). \textbf{(b)} We also show the synchronization level of the multiplex network, which is obtained directly through Eq. (\ref{eq:order_parameter_multilayer}) (solid gray line), and also through the composed solution Eq. (\ref{eq:order_parameter_composed}) (black dashed line).}
    \label{fig:trajectories_multiplex}
\end{figure*}
To highlight the applicability of the mathematical approach we introduce in this paper to study multiplex networks of nonlinear oscillators, we present here several examples and numerical simulations. \revision{In our examples, the inter-layer connectivity is given by a first-neighbors connection (or a $k$-regular graph with $k = 1$). This can be considered as a special case of our approach introduced in Sections \ref{sec:kuramoto_multiplex}, \ref{sec:composed solution} and \ref{sec:stability_analysis}, in which we consider $\epsilon_{lk} = 0$ for $|l-k| \neq 1$ in Eq. (\ref{eq:matrix_off-diag_block}) and  Eq. (\ref{eq:matrix_inter_system}).  The intra-layer connectivity varies across examples from random networks to $k$-regular graphs. A list of all the relevant parameters for the numerical simulations can be found in the Appendix.} We first consider a multiplex system composed of $M = 3$ layers with $N = 20$ oscillators each. In this simple example, we consider all oscillators with the same natural frequency and represent their motion in the rotating framework, i.e. $\omega = 0$. The initial state for the multiplex network is given by random phases in $\theta_{i}(0) \in \mathcal{U}(-\pi,\pi), i \in [1,NM]$. We then consider the intra-layer and the inter-layer systems, as described before, and use Prop. \ref{prop:composed solution} through Eq. (\ref{eq:solution_multiplex_composed}) to compose the initial conditions for one. With this, we can study the dynamics of these two smaller systems to learn about the dynamical behavior of the multiplex network. Figure \ref{fig:trajectories_multiplex}a shows the trajectories of the oscillators in the multiplex system (black solid lines) in each layer which are obtained directly through the numerical integration of Eq. (\ref{eq:KM_multiplex}) representing the multiplex system. We observe that these trajectories perfectly match the trajectories given by the composed solution or equivalent system (colored dashed lines), which are obtained through Eqs. (\ref{eq:KM_intra}), (\ref{eq:KM_inter}) and (\ref{eq:solution_multiplex_composed}). In this case, we show the trajectories of oscillators in each layer separately to help with the visualization, but we highlight that Eq. (\ref{eq:solution_multiplex_composed}) defining the composed solution is general and establishes the correspondence with the multiplex system. Here, we observe that the oscillators in all layers start at random phases and asynchronous states, due to the intra-layer and inter-layer coupling, though, the oscillators evolve to a common phase, characterizing phase synchronization.

We also analyze the dynamics of the multiplex network by using the Kuramoto order parameter, which gives us the level of phase synchronization of the system. We first use Eq. (\ref{eq:order_parameter_multilayer}) to obtain directly the order parameter through the phases $\theta_{i}(t)$, which are obtained through the integration of Eq. (\ref{eq:KM_multiplex}). This is represented by the gray solid line (multiplex system) in Fig. \ref{fig:trajectories_multiplex}b. We also obtain the Kuramoto order parameter with the composed solution (black dashed line). In this case, we can use Prop. \ref{prop:order_parameter} through Eq. (\ref{eq:order_parameter_composed}) to obtain the synchronization level of the whole system, but using information about the smaller ones, i.e. intra and inter-layer systems. We observe a complete match between these two lines, representing the transition from an asynchronous state to phase synchronization.

To further explore the dynamical behavior of multiplex networks and the use of the composed solution, we now consider the case where the natural frequency of oscillators is no longer zero. Figure \ref{fig:example_natural frequency}a shows the case where all oscillators have the same natural frequency. We then use the same procedure introduced in the previous figure and start the system with random initial conditions. We evaluate the synchronization level of the network, which is obtained directly through Eq.(\ref{eq:order_parameter_multilayer}) (gray solid line) and also using the composed solution Eq. (\ref{eq:order_parameter_composed}) (black dashed line), where we observe the network is asynchronous in the beginning, but it evolves to phase synchronization. We also analyze the spatiotemporal dynamics of the system, which is represented by the phases $\theta_{i}(t)$ plotted in color-code as a function of time and node index. We observe that, at first, each layer is desynchronized, but as time evolves emerge of phase synchronization is represented by the horizontal lines.
\begin{figure*}[htb]
    \centering
    \includegraphics[width=0.9\textwidth]{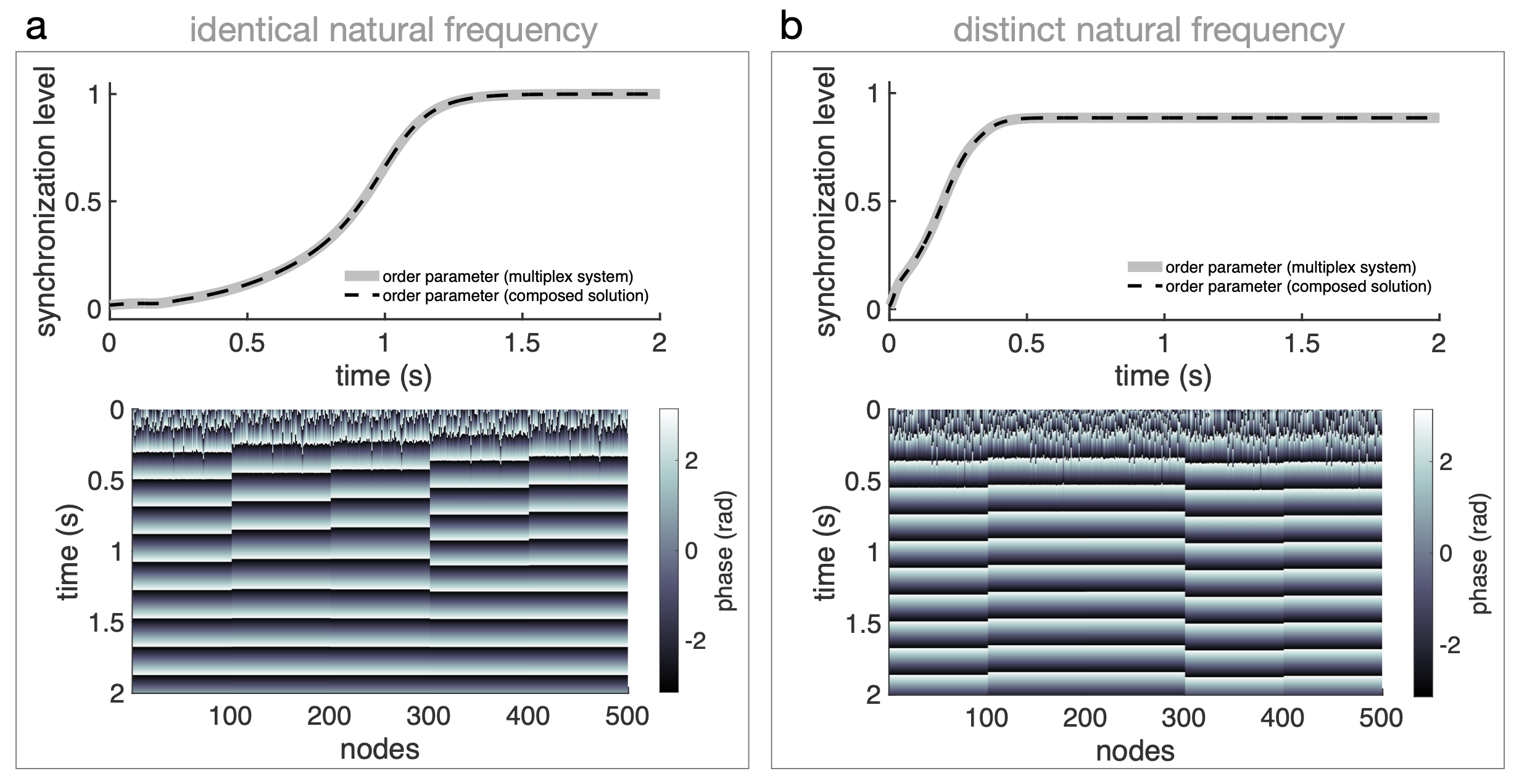}
    \caption{\textbf{Oscillation patterns in multiplex networks of Kuramoto oscillators.} \textbf{(a)} We first consider the case where all oscillators have the same, but non-zero natural frequency. We can measure the synchronization level using the order parameter obtained directly through Eq. (\ref{eq:order_parameter_multilayer}) -- gray line -- and also using the composed solution through Eq. (\ref{eq:order_parameter_composed}) -- black dashed line. The network starts in random initial phases and evolves to phase synchronization, which can be also observed in the spatiotemporal patterns given by the emergence of horizontal lines. \textbf{(b)} We also consider the case where oscillators in different layers have distinct natural frequencies. We observe that our approach produces a perfect match with the direct analysis of the multiplex network. Also, in this case, each layer is phase synchronized internally, but the multiplex system as a whole does not reach phase synchronization due to the distinct natural frequency, as observed in the spatiotemporal dynamics. }
    \label{fig:example_natural frequency}
\end{figure*}

Figure \ref{fig:example_natural frequency}b shows the case where the oscillators in different layers have distinct natural frequencies. In this more complicated case, our approach described in Prop \ref{prop:composed solution} through Eq. (\ref{eq:solution_multiplex_composed}) is still valid and we study the Kuramoto order parameter of the multiplex network (gray solid line) using the composed solution (black dashed line) \revisiontwo{-- see Remark~\ref{rem:natural_frequency}}. Again, due to random initial conditions, the systems start in an asynchronous state. Due to the intra and inter-layer couplings, however, the networks exhibit higher levels of coherence and synchronization. In this case, though, because the natural frequencies are distinct the order parameter does not reach one, and the multiplex network as a whole does not reach phase synchronization. We observe the spatiotemporal dynamics in this case and notice that each layer evolves to an individual state of phase synchronization within a given layer. Because each layer is oscillating at a different frequency, the system as a whole does not evolve to a common phase. It is important to emphasize that our approach offers a perfect match in this case as well. \revision{At this point, it is important to emphasize that the previous result depicted in Fig. \ref{fig:trajectories_multiplex} highlights the application of Prop. \ref{prop:composed solution}, which can be applied to oscillators with a homogeneous natural frequency. At the same time, the results in Fig. \ref{fig:example_natural frequency} show that we can apply Prop. \ref{prop:order_parameter} even to networks with heterogeneous natural frequencies in different layers, since the rotation speed of an oscillator does not affect the Kuramoto order parameter, which is defined by the difference in phase only.}

\begin{figure*}[t!]
    \centering
    \includegraphics[width=0.75\textwidth]{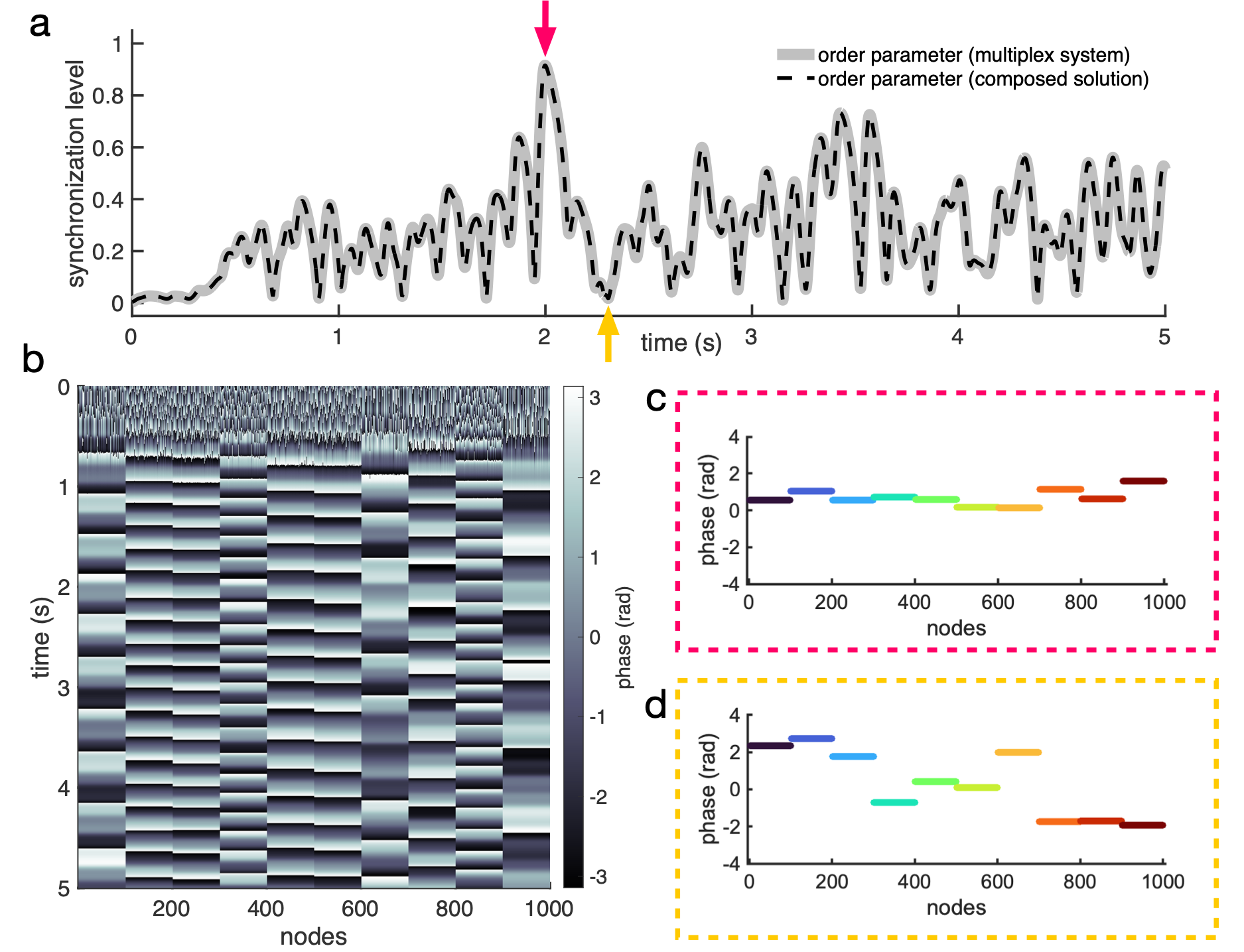}
    \caption{\textbf{Multiplex networks of oscillators with heterogeneous natural frequency display rich dynamics.} We consider a large multiplex network composed of $M = 10$ layers each one with $N = 100$ oscillators. In this case, the oscillators in different layers have distinct natural frequencies. \textbf{(a)} We observe that the synchronization level of the multiplex system no longer monotonically transitions to unity, but rather exhibits sophisticated dynamics, where the system transitions between different levels of coherence. Further, the order parameter obtained directly through Eq. (\ref{eq:order_parameter_multilayer}) depicts a perfect match with the one obtained using Prop. \ref{prop:order_parameter} through Eq. (\ref{eq:order_parameter_composed}). \textbf{(b)} We observe that each layer transitions to phase synchronization individually, but no phase synchronization is observed in the multiplex network as a whole. In this case, the network transitions between lower and higher levels of coherence among layers, where at specific points \textbf{(c)} the layers are in sync (pink arrow) \textbf{(d)} but it quickly transitions to a state with a low level of coherence (yellow arrow).}
    \label{fig:distinct_natural_frequency}
\end{figure*}
Multiplex networks composed of Kuramoto oscillators with heterogeneous natural frequency can display a rich diversity of synchronization patterns, and the approach we introduce in this paper is able to offer a simplified way to analyze them. To exemplify this point, we consider a larger multiplex network composed of $M = 10$ layers, each one with $N = 100$ oscillators. In this case, oscillators within a given layer have the same natural frequency, but are distinct from oscillators in another layer. As considered before, we can use Prop. \ref{prop:composed solution} to obtain the initial conditions for the multiplex, intra-layer and inter-layer networks in a way that we can use the smaller systems to obtain information about the multiplex one. Here, we use this proposition with random initial conditions $\mathcal{U}(-\pi, \pi)$. We then numerically integrate the dynamics of the multiplex network using Eq. (\ref{eq:KM_multiplex}), which allows us to obtain the Kuramoto order parameter for the whole system using directly Eq. (\ref{eq:order_parameter_multilayer}) (Fig. \ref{fig:distinct_natural_frequency}a, gray solid line). As we discussed before, we can use Prop. \ref{prop:order_parameter} through Eq. (\ref{eq:order_parameter_composed}) to obtain the level of synchronization that the multiplex through the composed solution (Fig. \ref{fig:distinct_natural_frequency}a, black dashed line), which perfectly matches the result obtained directly through the integration of the large multiplex system. In this case, we observe that, due to the interplay between the intra and inter-layer couplings and distinct natural frequencies, the dynamical behavior of the multiplex network is quite rich, and the order parameter does not increase monotonically to one. Instead, the synchronization level of the system is in constant change, increasing and decreasing in a sophisticated manner. 

Further, Fig. \ref{fig:distinct_natural_frequency}b shows the spatiotemporal dynamics of the multiplex system, where we can observe that each layer evolves to a phase synchronized state, but due to the distinct natural frequency of oscillation, the layers transition between higher and lower levels of coherence. This can be better appreciated in Figs. \ref{fig:distinct_natural_frequency}c and \ref{fig:distinct_natural_frequency}d, where we plot the phases of oscillators in each layer and consider two different moments. First, when the synchronization level of the multiplex network is high (pink arrow), we can observe that the layers have a high degree of coherence (Fig. \ref{fig:distinct_natural_frequency}c). However, when the level of synchronization of the whole system is low (yellow arrow), we can observe that layers are not synchronized (Fig. \ref{fig:distinct_natural_frequency}d).
\begin{figure*}[b!]
    \centering
    \includegraphics[width=\textwidth]{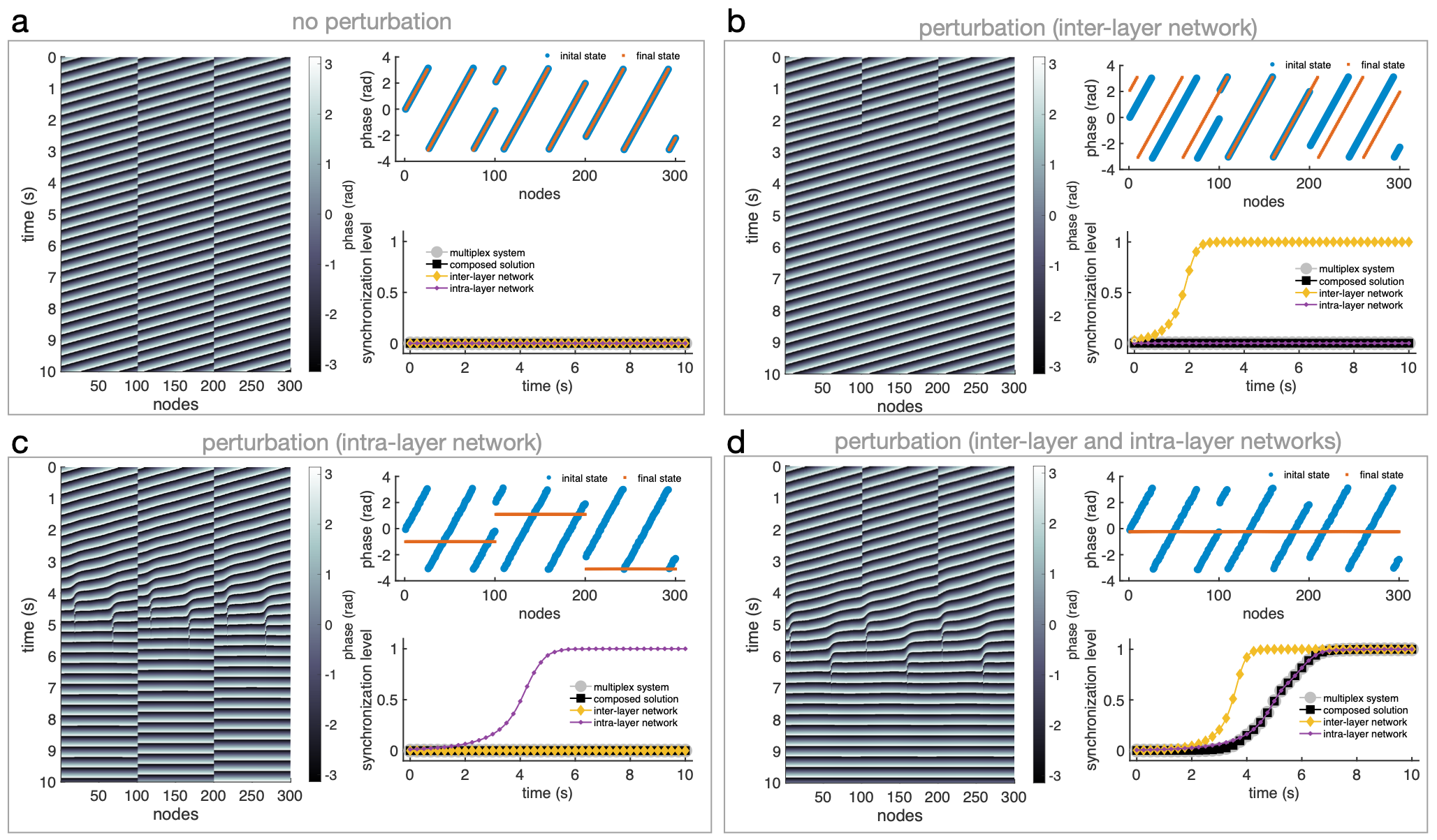}
    \caption{\textbf{Composing solutions for multiplex networks.} We use Prop. \ref{prop:composed solution} through Eq. (\ref{eq:solution_multiplex_composed}) to compose solutions for the multiplex system with $M = 3$ layers each one with $N = 100$ oscillators. Here, we consider individual solutions for the inter and intra-layer networks given by Eq. (\ref{eq:twisted_states}) with $p = 1$ and $p = 2$, respectively. \textbf{(a)} In the first example, without any perturbation, the system does change its pattern when we use this solution as the initial state. \textbf{(b)} We then consider perturbation on the inter-layer solution. In this case, the multiplex network transitions to a different twisted state, which is still a solution for the system. The inter-layer network transitions to phase synchronization due to perturbation. \textbf{(c)} When we perturb the solution of the intra-layer system, we observe a different transition in the spatiotemporal dynamics. In this case, each layer transition to phase synchronization, but on a different phase, which characterizes a new solution for the multiplex network with an order parameter equal to zero. \textbf{(d)} At last, we consider perturbation on the solution for the inter and intra-layer systems. In this case, the multiplex network as a whole transitions to phase synchronization, and the order parameter reaches one.}
    \label{fig:solutions_stability}
\end{figure*}

So far, we observe that by using the intra and inter-layer systems, we can study the dynamics of the respective multiplex network. In the cases presented in the previous figures, we start with random initial conditions and study the emergence of patterns of oscillations and synchronization. At the same time, however, we can use Prop \ref{prop:composed solution} through Eq. (\ref{eq:solution_multiplex_composed}) to find possible solutions for the multiplex network of Kuramoto oscillators. To do so, we can find solutions for the systems described by Kuramoto oscillators on the inter-layer and intra-layer matrices, and then use Eq. (\ref{eq:solution_multiplex_composed}) to compose a solution for the multiplex network. Further, we can then use Prop. \ref{prop:linear stability} to study the linear stability of these solutions, by considering the linear stability of the solutions for the inter and intra-layer networks. To exemplify this point, we consider multiplex networks where the intra-layer connection is given by a ring graph with periodic boundary conditions. In this case, possible solutions for the inter and intra-layer system can be expressed as:
\begin{equation}
    \bm{\theta}^{(p)}=\left(0,  \frac{-2 \pi p}{\mathcal{N}}, \cdots, \frac{-2 \pi p (\mathcal{N}-1)}{\mathcal{N}} \right),
\label{eq:twisted_states}
\end{equation}
where $\mathcal{N}$ is the number of oscillators, being $\mathcal{N} = M$ for the inter-layer network and $\mathcal{N} = N$ for the intra-layer network, and $p$ defines different solutions, where $p = 0$ represents phase synchronization, and $0<p\leq\sfrac{\mathcal{N}}{2}$ represents different phase-locking or twisted states. As an example, we consider the solutions for the inter and intra-layer networks using Eq. (\ref{eq:twisted_states}) with $p = 1$ and $p = 2$, respectively. In this case, we consider a multiplex network with $M = 3$ layers each one composed of $N = 100$ oscillators with the same natural frequency. With this, we have the solutions for the inter and intra-layer networks, $\bm{\phi}^{\ast}$ and $\bm{\psi}^{\ast}$, respectively. We then use Eq. (\ref{eq:solution_multiplex_composed}) to compose the solution for the multiplex network and use it as the initial condition for the numerical simulation. Figure \ref{fig:solutions_stability}a shows the numerical results for this system, where we observe that the spatiotemporal dynamics (left) stay the same as time evolves since the obtained phase configuration (upper right) is a solution for the multiplex network. This can be also appreciated by the evaluation of the Kuramoto order parameter for each network (lower right), where it is zero for the whole simulation. Here, the order parameter for the multiplex network is obtained directly through Eq. (\ref{eq:order_parameter_multilayer}), and the order parameter for the composed solution is given by Prop. \ref{prop:order_parameter} through Eq. (\ref{eq:order_parameter_composed}), the order parameter for the inter-layer system is given by Eq.(\ref{eq:order_parameter_inter}), and for the intra-layer network by Eq. (\ref{eq:order_parameter_intra}).

We can now consider perturbations over these solutions and study the response of the system. The perturbations are applied to each element of the vector representing the solution in Eq. (\ref{eq:twisted_states}). Mathematically, the perturbation is given by $\eta\mathcal{U}(-\pi, \pi)$, where $\eta = 0.025$ is the amplitude. After adding the perturbation to the solution given by Eq. (\ref{eq:twisted_states}), we wrap the phase pattern again in the interval of $[-\pi,\pi)$ to facilitate the visualization. We again consider the solutions for the inter and intra-layer networks using Eq. (\ref{eq:twisted_states}) with $p = 1$ and $p = 2$, respectively, with $M = 3$ layers each one composed of $N = 100$ oscillators. In Fig. \ref{fig:solutions_stability}b, we consider the case where we apply the perturbation to the solution of the inter-layer network. In a similar way observed before, the spatiotemporal dynamics (left) starts with each layer in distinct phase-locking states, but due to the perturbation on the inter-layer system, the multiplex network transitions to a different phase-locking state (upper right). This new state is still a solution for the multiplex system, and the order parameter for the multiplex network and for the composed solution does not change and is zero for the whole simulation (lower right). Interestingly, the order parameter for the inter-layer network transitions to one due to the perturbation, since the twisted state for this system is not stable ($M = 3$). The order parameter for the intra-layer networks does not change, otherwise.

Going further, we consider the case with perturbation on the intra-layer system (Fig. \ref{fig:solutions_stability}c). In this case, we observe a different transition in the spatiotemporal dynamics (left). Due to the perturbation in the solution for the intra-layer system, we can appreciate a small deviation in the initial state in comparison to the previous cases (upper right). Interestingly, because the perturbation occurs on the intra-layer level, each layer transition to phase synchronization, and the order parameter for the intra-layer system transition to one (lower right). At the same time, because the solution for the multiplex network is composed in combination with the inter-layer solution, the multiplex network as a whole does not transition to a common phase, and the order parameter of the multiplex network, as well as, of the composed solution, remains zero (lower right). In this case, the phase pattern for the whole system is now given by individual phase synchronized layers in different phases (upper right). At last, we consider the case where the perturbation is applied to the inter and intra-layer solutions (Fig. \ref{fig:solutions_stability}d). Under these conditions, the system starts with the combination of twisted states, but the whole multiplex network transitions to phase synchronization (left, upper right). This can be also observed using the order parameter, where the order parameter for the multiplex network, composed solution, inter-layer, and intra-layer networks transitions to one (lower right). We emphasize the correspondence between the order parameter obtained directly through Eq. (\ref{eq:order_parameter_multilayer}) for the multiplex network, and through the composed solution using Eq. (\ref{eq:order_parameter_composed}).

\section{Discussions and conclusions}\label{sec:discussion_conclusion}

Networked systems have been extensively studied in recent years, where many direct applications for real-world problems are possible. In this paper, we have studied multiplex networks, which are a class of multilevel systems. There has been great interest in this class of networks recently \cite{kivela2014multilayer,salehi2015spreading,menichetti2016control,de2016physics,aleta2019multilayer,bassett2017network,yuvaraj2021topological,gambuzza2015intra,leyva2018relay,leyva2017inter,singh2015synchronization,schulen2021solitary,gao2012networks,ji2023signal}. At the same time, these networks have an intrinsic difficulty regarding mathematical and analytical approaches due to the presence of two different scales of coupling. A common technique used to study these systems consists in dividing the whole, multilevel system into smaller representations of the system, e.g. the intra and inter-coupling, which represent the coupling within a layer and between layers, respectively. In this regard, many researchers have used this kind of approach to study multilayer networks \cite{kivela2014multilayer,de2013mathematical,boccaletti2014structure,gao2012networks,berner2021multiplex}. 

In this paper, we have contributed to this discussion by introducing an approach to multiplex networks of Kuramoto oscillators. This framework allows us to describe the oscillation patterns observed in coupled Kuramoto oscillators on a multiplex network by studying two smaller, simpler systems of Kuramoto oscillators related to the intra and inter-coupling structures. With this in mind, we can compose solutions and oscillation patterns in multiplex networks of Kuramoto oscillators in a straightforward way, as shown in Sec. \ref{sec:composed solution}. This approach allows us to (i) study the dynamics of a whole multiplex system of oscillators by analyzing the simpler systems represented by the intra and inter-couplings, and (ii) find equilibrium points for a sophisticated multiplex network by composing a solution based on the smaller intra and inter systems. This mathematical approach is defined in Props. \ref{prop:composed solution} and \ref{prop:order_parameter}, which have been explored in numerical simulations, as shown in Figs. \ref{fig:trajectories_multiplex}, \ref{fig:example_natural frequency}, and \ref{fig:distinct_natural_frequency}. Furthermore, we can use a similar idea to extend the approach to study the linear stability of solutions in the multiplex network of Kuramoto oscillators. Because we can represent and study the dynamical behavior of a multiplex network using the intra and inter-coupling systems, we can obtain information about the spectral properties of the Jacobian related to the multiplex system by studying the Jacobian of the smaller systems. With this, we can analyze the linear stability of the solutions for the intra and inter-coupling system, and then obtain information about the linear stability of the composed solution on the multiplex network. This is represented in Prop. \ref{prop:linear stability} and studied numerically in Fig. \ref{fig:solutions_stability}. \revision{Importantly, this mathematical approach naturally generalizes for multiplex networks of any size (considering the number of oscillators in each layer and/or the number of layers).}

The framework and results introduced in this paper contribute to the study of the dynamical behavior of networks with different levels of connection. Particularly, we have recently introduced a similar approach to multilayer networks \cite{nguyen2023broadcasting}, which allows us to study the dynamics of multilayer networks using a reduced representation. Moreover, a similar approach has been introduced recently for the study of multilayer networks, where the reduced representations of multilevel networks are used to study the master stability function of multilayer networks with application in neural networks \cite{berner2021multiplex}. Here, we have used similar ideas to extend the study of Kuramoto oscillators on multiplex networks, an important dynamical system that has extensively studied \cite{kumar2021explosive,jalan2019explosive,frolov2018macroscopic,khanra2018explosive,zhang2015explosive}. \revision{The framework we introduce here differs from previous work due to the fact we can apply it directly to oscillators with distinct natural frequencies (Figs. \ref{fig:example_natural frequency} and \ref{fig:distinct_natural_frequency}). In addition to that, this approach allows us to compose specific solutions for multiplex networks, and also to study a diversity of spatiotemporal patterns in these networks.} \revision{It is important to notice that our approach naturally generalizes for systems with different coupling functions than the sine function as described in the Kuramoto model.} Further, our approach can be extended in the future in light of a new perspective for nonlinear oscillator networks \cite{budzinski2022geometry,nguyen_equilibria_2023,budzinski2023analytical}, where analytical and geometric insights are available for the emergence of synchronization phenomena. The development of different mathematical approaches for networked systems is of great importance due to its direct applicability to study many problems. For instance, many problems in neuroscience can be modeled and studied using multilayer networks \cite{palmigiano2017flexible,avena2018communication,gosak2018network}, where the emergence of sophisticated spatiotemporal patterns plays an important role \cite{muller2018cortical,palmigiano2017flexible}. The study of these networks thus opens the possibility to develop further analyses of computations performed with these systems. 

\begin{acknowledgments}
We thank Alex Busch for helping with the illustrations. This work was supported by BrainsCAN at Western University through the Canada First Research Excellence Fund (CFREF), the NSF through a NeuroNex award (\#2015276), the Natural Sciences and Engineering Research Council of Canada (NSERC) grant R0370A01, Compute Ontario (computeontario.ca), Digital Research Alliance of Canada (alliancecan.ca), and the Western Academy for Advanced Research. R.C.B gratefully acknowledges the Western Institute for Neuroscience Clinical Research Postdoctoral Fellowship. 
\end{acknowledgments}

\section*{Code availability}

An open-source code repository for this work is available on GitHub: \href{http://mullerlab.github.io}{\textcolor{Cerulean}{http://mullerlab.github.io}}.



\section*{Appendix - computational details}

Numerical integration was performed with the Euler method with a time step given by $dt = 0.001$. The network structure is given as follows (if not stated otherwise): the connection scheme within each layer is given by an undirected Erdös-Rényi graph with a probability of $0.25$; the connection scheme between layers is given by a $k$-regular graph with $k = 1$. The parameters used in the simulations depicted in each figure are listed in Table \ref{table_parameters}.
\begin{table}[htb]
    \centering
    \caption{Relevant parameters for the computational analyses.}
    \begin{tabular}{l  l}
         \hline \hline
        Figure 2 & $M = 3$, $N = 20$, $\epsilon_{\mathrm{intra}} = 1$ and $\epsilon_{\mathrm{inter}} = 0.5$ \\
        Figure 3a & $M = 5$, $N = 100$, $\epsilon_{\mathrm{intra}} = 1$ $\epsilon_{\mathrm{inter}} = 3$ \\
        Figure 3b & $M = 5$, $N = 100$, $\epsilon_{\mathrm{intra}} = 1$ $\epsilon_{\mathrm{inter}} = 20$ \\
        Figure 4 & $M = 10$, $N = 100$, $\epsilon_{\mathrm{intra}} = 0.75$, $\epsilon_{\mathrm{inter}} = 10$ \\
        \multirow{2}{*}{Figure 5} & $M = 3$, $N = 100$, $\epsilon_{\mathrm{intra}} = 0.10$, $\epsilon_{\mathrm{inter}} = 1$ \\
         & $\bm{A}_{\mathrm{intra}}$ is given by $k$-regular graph with $k = 20$ \\
        \hline \hline
        \end{tabular}
        \label{table_parameters}
\end{table}


%

\end{document}